\documentclass[openright, 12pt]{article}
\usepackage{amsmath,amsfonts}
\usepackage{amsthm}
\newtheorem{thm}{Theorem}
\newtheorem{prop}[thm]{Proposition}
\newtheorem{lem}[thm]{Lemma}
\newtheorem{rem}[thm]{Remark}
\newtheorem{cor}[thm]{Corollary}

\newtheorem{defn}[thm]{Definition}

\newtheorem{qu}[thm]{Question}
\usepackage{float} 
\usepackage{multirow} 
\usepackage[dvips]{graphicx}
\usepackage{psfrag}
\usepackage{enumerate}

\usepackage{geometry}
\usepackage{amssymb}
\usepackage[latin1]{inputenc}

\begin{document}

\begin{center}

\noindent  \large First Betti number and collapse\normalsize

\begin{normalsize}
Sergio Zamora

Max Planck Institute for Mathematics

zamora@mpim-bonn.mpg.de
\end{normalsize}
\end{center}

\begin{abstract}
We show that when a sequence of Riemannian manifolds collapses under a lower Ricci curvature bound, the first Betti number cannot drop more than the dimension. 
\end{abstract}

\section{Introduction}

For $n \in \mathbb{N}$, $ c \in \mathbb{R}$, $ D  >  0 $, let $\mathfrak{M}_{Ric}(n,c,D)$ (resp. $\mathfrak{M}_{sec}(n,c,D)$) denote the class of closed $n$-dimensional Riemannian manifolds of Ricci curvature $\geq c$ (resp. sectional curvature $\geq c$) and diameter $\leq D$. A significant proportion of the subject consists of understanding the relationship between sequences $X_i \in \mathfrak{M}_{Ric}(n,c,D)$ and their Gromov--Hausdorff limits. Our main result concerns the first Betti number of such limit space.

\begin{thm}\label{main}
\rm Let $X_i \in \mathfrak{M}_{Ric}(n,c,D)$ be a sequence with $\beta_1(X_i) \geq r$ for each $i$. If $X_i$ converges in the Gromov--Hausdorff sense to a space $X$ containing a $k$-regular point, then
\[  \beta_1(X) \geq r + k - n.  \] 
\end{thm}

It has been known that for a  Riemannian manifold $M$ of almost non-negative Ricci curvature, if its first Betti number equals its dimension then $M$ is homeomorphic to a torus. This result has been recently extended to singular spaces by Mondello, Mondino, and Perales \cite{mondello-mondino-perales}. A consequence of their work and Theorem \ref{main} is the following.

\begin{cor}\label{tori}
\rm For each $n \in \mathbb{N}$, there is $\varepsilon > 0 $ such that if $X_i \in \mathfrak{M}_{sec}(n,-\varepsilon ,1\\
)$ is a sequence of spaces with $\beta_1(X_i )\geq n$ that converges in the Gromov--Hausdorff sense to a space $X$ of Hausdorff dimension $k$, then $X$ is bi-H\"{o}lder homeomorphic to a flat $k$-dimensional torus.
\end{cor}

\begin{rem}
\rm Theorem \ref{main} shows that the first Betti number cannot drop more than the dimension. Contrastingly, the fundamental group can decrease in the limit even if there is no collapse: Otsu has constructed a sequence of metrics in $\mathbb{S}^3 \times \mathbb{R}P^2$ of positive Ricci curvature that converges in the Gromov--Hausdorff sense to a simply connected 5-dimensional space \cite{otsu}.
\end{rem}

Theorem \ref{main} is an improvement of the main result of \cite{Z}. On the other hand, the goal of this program is to solve following problem.

\begin{qu}\label{tori-conjecture}
\rm Assume a sequence $X_i \in \mathfrak{M}_{Ric}(n,c,D ) $ of spaces homeomorphic to the $n$-dimensional torus converges in the Gromov--Hausdorff sense to a space $X$. Is $X$ necessarily homeomorphic to a torus?
\end{qu}

The author would like to thank Anton Petrunin for suggesting Question \ref{tori-conjecture} and motivating the author to work on this problem, Raquel Perales and Guofang Wei for interesting discussions and helpful comments, and the Max Planck Institute for Mathematics for its financial support and hospitality.

\section{Preliminaries}

In this section we recall the required material for Theorem \ref{main} and Corollary \ref{tori}, which we prove in the following section.

\subsection{Gromov--Hausdorff topology}

The basics on the subject can be found in (\cite{BBI}, Chapter 7). 
\begin{defn}
\rm We say that a function $f: X \to Y$ between metric spaces is an $\varepsilon$-isometry if for all $x_1, x_2 \in X$ one has $ \vert d^X(x_1,x_2) - d^Y(fx_1, fx_2) \vert \leq \varepsilon $, and $f(X)$ intersects each closed ball of radius $\varepsilon$ in $Y$. We say that a sequence of functions $f_i : X_i \to Y_i$ between metric spaces are \textit{Gromov--Hausdorff approximations} if $f_i$ is an $\varepsilon_i $-isometry for some sequence $\varepsilon_i \to 0$.  
\end{defn}

\begin{prop}\label{GH-convergence}
\rm (Gromov) Let $X_i$ be a sequence of compact metric spaces, and let $X$ be a complete metric space. Then the following are equivalent:
\begin{itemize}
\item There is a sequence $f_i : X_i \to X$ of Gromov--Hausdorff approximations.
\item There is a sequence $h_i : X \to X_i$ of Gromov--Hausdorff approximations.
\end{itemize}
In either case, $X$ is compact and one says that the sequence $X_i$ \textit{converges to $X$ in the Gromov--Hausdorff sense}. Furthermore, there is a metric on the class of compact metric spaces modulo isometry that yields this topology.
\end{prop}

\begin{defn}
\rm We say that a function $f: (X,x) \to (Y,y)$ between  pointed metric spaces is an $\varepsilon$-isometry if $fx=y$, for all $x_1, x_2 \in B^X(x, 2 / \varepsilon )$ one has $ \vert d^X(x_1,x_2) - d^Y(fx_1, fx_2) \vert \leq \varepsilon $, and $f(B^X(x, 2 / \varepsilon ) )$ intersects each closed ball of radius $\varepsilon$ in $B^Y(y, 1 / \varepsilon )$. We say that a sequence of functions $f_i : ( X_i ,x_i) \to  (Y_i, y_i ) $ between pointed metric spaces are \textit{pointed Gromov--Hausdorff approximations} if $f_i$ is a pointed $\varepsilon_i $-isometry for some sequence $\varepsilon_i \to 0$.  
\end{defn}

\begin{prop}\label{pointed-GH-convergence}
\rm (Gromov) Let $(X_i,x_i)$ be a sequence of proper pointed metric spaces, and let $(X,x)$ be a complete pointed metric space. Then the following are equivalent:
\begin{itemize}
\item There is a sequence $f_i : (X_i ,x_i) \to (X,x)$ of pointed Gromov--Hausdorff approximations.
\item There is a sequence $h_i : (X,x) \to (X_i,x_i)$ of pointed Gromov--Hausdorff approximations.
\end{itemize}
In either case, $X$ is proper and one says that the sequence $(X_i,x_i)$ \textit{converges to $(X,x)$ in the pointed Gromov--Hausdorff sense}. Furthermore, there is a metric on the class of proper pointed metric spaces modulo isometry that yields this topology.
\end{prop}

For $n \in \mathbb{N}$, $c \in \mathbb{R}$, we denote by $\mathfrak{M}_{Ric} (n,c)$ the class of complete $n$-dimensional Riemannian manifolds of Ricci curvature $\geq c$. One reason we know so much about these families of spaces is because they are pre-compact with respect to the Gromov--Hausdorff topology.

\begin{thm}\label{compactness}
\rm (Gromov) Let $(Y_i,y_i)$ be a sequence with $Y_i \in \mathfrak{M}_{Ric}(n,c) $ for each $i$. Then one can find a subsequence that converges in the pointed Gromov--Hausdorff sense to some proper metric space $(Y,y)$.
\end{thm}

\subsection{Equivariant Gromov--Hausdorff convergence}

There is a well studied notion of convergence of group actions in this setting. For a proper metric space $X$, the topology that we use on its group of isometries $Iso(X)$ is the compact-open topology, which in this setting coincides with both the topology of pointwise convergence and the topology of uniform convergence on compact sets. This topology makes $Iso(X)$ a locally compact second countable metrizable group.

\begin{defn}
\rm Let $(Y_i,q_i) $ be a sequence of proper metric spaces that converges in the pointed Gromov--Hausdorff sense to a proper space $(Y,q)$. Consider pointed Gromov--Hausdorff approximations $f_i : (Y_i,q_i ) \to (Y,q) $ and $h_i: (Y,q) \to (Y_i,q_i)$ such that $d^Y(f_i \circ h_i(y),y)\to 0$ for all $y \in Y$. Also let $\Gamma_i \leq Iso(Y_i)$ be a sequence of groups of isometries. We say that $\Gamma_i$ \textit{converges in the equivariant Gromov--Hausdorff sense to} a closed group $\Gamma \leq Iso (Y)$ if for all $R, \varepsilon > 0 $, one has the following:
\begin{itemize}
\item For each $g \in \Gamma$, there is $ i_0 \in \mathbb{N}$ such that for each $i \geq i_0$ there is $g_i \in \Gamma_i$ with $d^Y ( f_i \circ g_i \circ h_i (y), g(y)) \leq \varepsilon $ for all $y \in B^{Y}(q,R )$.
\item There is $i_0\in \mathbb{N}$ such that if $i \geq i_0$, $g \in \Gamma_i$ with $d^Y(gq_i,q_i)\leq R$, then there is $\gamma \in \Gamma $ such that $d ^Y( f_i \circ g \circ h_i (y), \gamma (y)) \leq \varepsilon $ for all $y \in B^{Y}(q,10R)$.
\end{itemize}
Although this definition clearly depends on $f_i$ and $h_i$, we usually omit this when we state that $\Gamma_i$ converges to $\Gamma$.
\end{defn}

This definition of equivariant convergence allows one to take limits before or after taking quotients.

\begin{lem}\label{equivariant}
\rm Let $(Y_i,q_i)$ be a sequence of proper metric spaces that converges in the pointed Gromov--Hausdorff sense to a proper space $(Y,q)$, and  $\Gamma_i \leq Iso(Y_i)$ a sequence of isometry groups that converges in the equivariant Gromov--Hausdorff sense to a closed group $\Gamma \leq Iso (Y)$. Then the sequence $(Y_i/\Gamma_i, [q_i])$ converges in the pointed Gromov--Hausdorff sense to $(Y/\Gamma , [q])$.
\end{lem}

Since the isometry groups of proper metric spaces are locally compact, one has an Arzel\'a-Ascoli type result (\cite{fukaya-yamaguchi}, Proposition 3.6).

\begin{thm}\label{equivariant-compactness}
\rm (Fukaya--Yamaguchi) Let $(Y_i,q_i) $ be a sequence of proper metric spaces that converges in the pointed Gromov--Hausdorff sense to a proper space $(Y,q)$, and take a sequence $\Gamma_i \leq Iso(Y_i)$ of groups of isometries. Then there is a subsequence $(Y_{i_k}, q_{i_k}, \Gamma_{i_k})_{k \in \mathbb{N}}$ such that $\Gamma_{i_k}$ converges in the equivariant Gromov--Hausdorff sense to a closed group $\Gamma \leq Iso(Y)$.
\end{thm}

In \cite{gromov-afm}, Gromov studied which is the structure of discrete groups that act transitively on spaces that look like $\mathbb{R}^n$. Using the Malcev embedding theorem, he showed that they look essentially like lattices in nilpotent Lie groups. In \cite{breuillard-green-tao}, Breuillard--Green--Tao studied in general what is the structure of discrete groups that have a large portion acting on a space of controlled doubling. It turns out that the answer is still essentially just lattices in nilpotent Lie groups. In (\cite{zamora-lahs}, Sections 7-9) the ideas from \cite{gromov-afm} and \cite{breuillard-green-tao} are used to obtain the following structure result.

\begin{thm}\label{zamora}
\rm Let $(Z,p)$ be a proper pointed geodesic space of topological dimension $\ell \in \mathbb{N}$ and let $(D_i , p_i )$ be a sequence of discrete metric spaces converging in the pointed Gromov--Hausdorff sense to $(Z,p)$. Assume there is a sequence of isometry groups $\Gamma_i \leq Iso (D_i)$ that act transitively and for each $i$, $\Gamma_i$ is generated by its elements that move $p_i$ at most $10$. Then for large enough $i$, there are finite index subgroups $G_i \leq \Gamma_i$ and finite normal subgroups $F_i \triangleleft G_i  $ such that $G_i / F_i $ is isomorphic to  a quotient of a lattice in a nilpotent Lie group of dimension $\ell$. In particular, if the groups $\Gamma_i$ are abelian, for large enough $i$ their rank is at most $\ell$. 
\end{thm}

For $k \in \mathbb{N}$, a proper metric space $X$, we say that $x \in X$ is a $k$\textit{-regular point} if for any sequence $\lambda_i \to \infty$, the sequence $(\lambda_i X, x )$ converges in the pointed Gromov--Hausdorff sense to $\mathbb{R}^k$. For limits of sequences in $\mathfrak{M}_{Ric}(n,c)$, almost all points are regular \cite{cheeger-colding}.

\begin{thm}\label{CC}
\rm  (Cheeger--Colding) Let $ X_ i \in \mathfrak{M}_{Ric}(n,c) $ converge in the pointed Gromov--Hausdorff sense to a space $X$. If $\mathcal{R}_k$ denotes the set of $k$-regular points of $X$, then $\mathcal{R}_k \neq \emptyset$ implies $k \leq n$, and $\cup_{j=0}^n \mathcal{R}_k$ is dense in $X$.
\end{thm}

Arguably the most used tool in the theory of Riemannian manifolds of non-negative Ricci curvature is the Cheeger--Gromoll splitting theorem. It was later generalized by Cheeger and Colding to limits of Riemannian manifolds \cite{cheeger-colding-splitting}. Using this, one could understand how $\mathbb{R}^{k}$ arises as a quotient of such spaces. 

\begin{thm} \label{cc-split}
\rm (Cheeger--Colding) Let $\varepsilon_i \to 0$ and $(Y_i, q_i) \in \mathfrak{M}_{Ric}(n,- \varepsilon_i) $ a sequence that converges in the pointed Gromov--Hausdorff sense to $(Y,q)$. If $Y$ contains an isometric copy of $\mathbb{R}^k$, then $Y$ split as a metric space as $\mathbb{R}^k \times Z$ for some proper geodesic space $Z$ of Hausdorff dimension $\leq n-k$.
\end{thm}

\begin{cor}\label{splitting}
\rm  Let $\varepsilon_i \to 0$ and $(Y_i, q_i) \in \mathfrak{M}_{Ric}(n,- \varepsilon_i) $ be a sequence that converges in the pointed Gromov--Hausdorff sense to $(Y,q)$. Assume there is a sequence of groups of isometries $\Gamma_i \leq Iso (Y_i)$ such that $(Y_i/\Gamma_i , [q_i])$ converges in the pointed Gromov--Hausdorff sense to $\mathbb{R}^{k}$ and $\Gamma_i$ converges in the equivariant Gromov--Hausdorff sense to a group $\Gamma \leq Iso (Y)$. Then $Y$ splits as a metric space as $\mathbb{R}^{k}\times Z$ for some proper geodesic space $Z$ of Hausdorff dimension $\leq n- k$, and the $Z$-fibers given by this product coincide with the orbits of $\Gamma$.
\end{cor}

\begin{proof}
One can use the submetry $\phi: Y \to  Y/\Gamma  = \mathbb{R}^k$ to lift the lines of $\mathbb{R}^k$ to lines in $Y$ passing through $q$. By Theorem \ref{cc-split},  we get the desired splitting $Y = \mathbb{R}^k \times Z$ with $\phi (z_0, x)=x $ for all $x \in \mathbb{R}^k$ and some $z_0 \in Z$. 

Let $g \in \Gamma$ and assume $g(z_0,x ) = (z,y)$ for some $z_0,z \in Z$, $x,y \in \mathbb{R}^k$. Then for all $t \geq 1$, one has
\begin{eqnarray*}
t \vert y-x \vert & = & \vert \phi (z_0, x+ t(y-x)  - \phi ( (z_0,x)) \vert \\
& = & \vert  \phi  (z_0, x+ t(y-x))- \phi (z,y)  \vert        \\
& \leq & d^Y (  (z_0, x + t (y-x)) , (z,y) )\\
& = & \sqrt{  d^Z(z_0,z)^2 +  \vert  (t-1) (y-x) \vert ^2 }.
\end{eqnarray*} 
As $t \to \infty$, this is only possible if $x=y$, and we conclude that the action of $\Gamma$ respects the splitting $Y = \mathbb{R}^k \times Z$.
\end{proof}

\subsection{Homology and Ricci curvature bounds}

We define the \textit{content} of a map $A \to X$ between topological spaces to be the image of the natural map $H_1(A) \to H_1(X)$. If $\mathcal{U}$ is a family of subsets of $X$, we  denote by $H_1(\mathcal{U} \prec X ) \leq H_1(X)$ the subgroup generated by the contents of the inclusions $U \to X$ with $U \in \mathcal{U}$. This group satisfies a natural monotonicity property.

\begin{lem}\label{refinement}
\rm Let $X$ be a topological space, and $\mathcal{U}$, $\mathcal{V}$ two families of subsets of $X$. If for each $U \in \mathcal{U}$ there is $V \in \mathcal{V}$ with $U \subset V$, then $H_1(\mathcal{U}\prec X) \leq H_1(\mathcal{V} \prec X)$.
\end{lem}

If $\varepsilon > 0 $, $X$ is a metric space, and $\mathcal{U}$ is the family of balls of radius $\varepsilon$ in $X$, then we denote $H_1(\mathcal{U}\prec X)$ simply by $H_1^{\varepsilon}( X)$. It has been recently shown that limits of sequences in $\mathfrak{M}_{Ric}(n,c,D)$ are semi-locally-simply-connected \cite{wang}.

\begin{thm}\label{slsc}
\rm (Pan--Wang) Let $X_i \in \mathfrak{M}_{Ric}(n,c,D)$ converge in the Gromov--Hausdorff sense to a space $X$. Then $X$ is semi-locally-simply-connected. In particular, $H^{\varepsilon}_1(X)$ is trivial for small enough $\varepsilon$.
\end{thm}

\begin{thm}\label{SW}
\rm (Sormani--Wei) Let $X$ be a compact geodesic space. Assume there is $\varepsilon > 0 $ such that $H_1^{2\varepsilon}( X)$ is trivial, and let $Y$ be a compact geodesic space with $f:Y \to X$ an $\varepsilon /100 $-approximation. Then there is a surjective morphism $ H_1 (Y) \to H_1 (X)$ (independent of $\varepsilon$) whose kernel is precisely $H_1^{\varepsilon}(Y)$.
\end{thm}

\begin{proof}[Proof sketch:]
We follow the lines of (\cite{SW}, Theorem 2.1), where they prove this result for $\pi_1$ instead of $H_1$. Each 1-cycle in $Y$ can be thought as a family of loops $\mathbb{S}^1 \to Y$ with integer multiplicity. For each map $ \gamma : \mathbb{S}^1 \to Y$, by uniform continuity one could pick finitely many cyclically ordered points $\{ z_1, \ldots , z_m \} \subset \mathbb{S}^1$ such that $\gamma ([z_{j-1}, z_j])$ is contained in a ball of radius $\varepsilon /10$ for each $j$. Then set $\phi ( \gamma ) : \mathbb{S}^1 \to X $ to be the loop with $\phi (\gamma )(z_j) = f( \gamma (z_j) )$ for each $j$, and $\phi (\gamma ) \vert _{[z_{j-1}, z_j]}$ a minimizing geodesic from $\phi (\gamma )(z_{j-1})$ to $\phi (\gamma )(z_{j})$. 

Clearly, $\phi (\gamma )$ depends on the choice of the points  $z_j$ and the minimizing paths $\phi (\gamma ) \vert _{[z_{j-1}, z_j]}$. However, the homology class of $\phi (\gamma )$ in $H_1(X)$ does not depend on these choices, since different choices yield curves that are $\varepsilon$-uniformly close, which by hypothesis are homologous. 

Assume that a 1-cycle $c$ in $Y$ is the boundary $\partial \sigma$ of a 2-chain $\sigma$. After taking iterated barycentric subdivision, one could assume that each simplex of $\sigma$ is contained in a ball of radius $\varepsilon /10$. By recreating $\sigma $ in $X$ via $f$ simplex by simplex, one could find a 2-chain whose boundary is $\phi (c)$. This means that $\phi$ induces a map $ \tilde{\phi}:H_1(Y) \to H_1(X)$.

In a similar fashion, if a 1-cycle $c$ in $Y$ is such that $\phi (c) $ is the boundary of a 2-chain $\sigma$, one could again apply iterated barycentric subdivision to obtain a 2-chain $\sigma^{\prime}$ in $X$ whose boundary is $\phi (c)$ and such that each simplex is contained in a ball of radius $\varepsilon /10$. Using $f$ one could recreate the 1-skeleton of $\sigma ^{\prime}$ in $Y$ in such a way that expresses $c$ as a linear combination with integer coefficients of 1-cycles contained in balls of radius $\varepsilon$ in $Y$. This implies that the kernel of $\tilde{\phi} $ is contained in $H_1^{\varepsilon}(Y)$.

If a 1-cycle $c$ in $Y$ is contained in a ball of radius $\varepsilon$, then $\phi (c)$ is contained in a ball of radius $2 \varepsilon$ and then by hypothesis, $\phi (c)$ is a boundary. This shows that the kernel of $\tilde{\phi}$ is precisely $H_1^{\varepsilon}(Y)$.

Lastly, for any loop $\gamma : \mathbb{S}^1 \to X$, one can create via $f$ a loop  $\gamma_1 : \mathbb{S}^1 \to Y$ such that $\phi ( \gamma_1)$ is uniformly close (and hence homologous) to $\gamma$, so $\tilde{\phi}$ is surjective. 
\end{proof}

\begin{cor}\label{SW-gap}
\rm Let $X$ be a compact geodesic space. Assume there is $\rho > 0 $ such that  $H_1^{2\rho}( X)$ is trivial, and consider a sequence $X_i$ of compact geodesic spaces that converges to $X$ in the Gromov--Hausdorff sense. Then there is a sequence $\rho _i \to 0$ such that $ H_1^{\rho_i}( X_i) = H_1^{\rho}( X_i)  $ for each $i$.
\end{cor}

\begin{proof}
For large enough $i$, let $\rho_i \in  ( 0 , \rho ] $ be such that $\rho_i \to 0$ and there is a $\rho_i/100$-approximation $X_i \to X $. One could then apply Theorem \ref{SW} for $\varepsilon \in [\rho_i , \rho]$ to get a map $H_1(X_i) \to H_1(X)$ whose kernel equals both  $ H_1^{\rho_i}(X_i) $ and $ H_1^{\rho}(X_i) $. For small $i$, simply set $\rho_i = \rho$.
\end{proof}

The following results were obtained in \cite{kapovitch-wilking}, and are stated in terms of $\pi_1$. The first one states that for $M \in \mathfrak{M}_{Ric}(n,c,D)$, there is a subgroup $N \leq H_1(M) $ that can be detected anywhere. The second one states that at regular points, there is a gap phenomenon.

\begin{thm}\label{KW}
\rm  (Kapovitch--Wilking) For each $n \in \mathbb{N}$, $c \in \mathbb{R}$, $D > 0 $, $\varepsilon_1 > 0$, there are $\varepsilon_0 > 0$, $C \in \mathbb{N}$, such that the following holds. For each $ M \in \mathfrak{M}_{Ric}(n,c,D) $, there is $\varepsilon \in [\varepsilon_0, \varepsilon_1]$ and a subgroup $N \leq H_1(M)$ such that for all $x \in M$, 
\begin{itemize}
\item $N$ lies in the content of the inclusion $ B^M(x, \varepsilon / 1000 )  \to M$.
\item The index of $N$ in the  content of the inclusion $ B^M(x, \varepsilon )  \to  M  $ is $\leq C$.
\end{itemize}
\end{thm}

\begin{lem}\label{KW-gap}
\rm (Kapovitch--Wilking) Let $X_i \in \mathfrak{M}_{Ric}(n,c,D)$ converge in the Gromov--Hausdorff sense to a space $X$. Consider a $k$-regular point $x \in X$, and $h_i : X \to X_i$ a sequence of Gromov--Hausdorff approximations. Then there is $\eta > 0 $ and a sequence $\eta_i \to 0$ such that the contents of the inclusions $ B^{X_i}(h_i(x), \eta_i)  \to  X_i $, $ B^{X_i}(h_i(x), \eta)  \to  X_i $ coincide.
\end{lem}

For the proof of Corollary \ref{tori} we require the following result from \cite{mondello-mondino-perales}.

\begin{thm}\label{torus-stability}
\rm (Mondello--Mondino--Perales) For each $n \in \mathbb{N}$ there is $\varepsilon > 0 $ such that if $X_i \in \mathfrak{M}_{sec}(n,-1,\varepsilon ) $ converges in the Gromov--Hausdorff sense to a space $X$ of Hausdorff dimension $k$ and $\beta_1(X) \geq k$, then $X$ is bi-H\"{o}lder homeomorphic to a flat $k$-dimensional torus.
\end{thm}

\section{Proof of the main results}

\begin{proof}[Proof of Theorem \ref{main}:]
Let $p \in X$ be a $k$-regular point, $h_i : X \to X_i$ a sequence of Gromov--Hausdorff approximations, and set $p_i : = h_i (p)$. Then by Theorem \ref{KW-gap}, there is $\varepsilon_2 > 0 $ and a sequence $\eta_i \to 0$ such that the contents of the maps  $ B^{X_i}(p_i, \eta_i)  \to X_i$, $ B^{X_i}(p_i , \varepsilon_2) \to  X_i$ coincide.

By Theorem \ref{slsc}, there is  $ \varepsilon_1 \in (0, \varepsilon_2 ] $ such that for each $x \in X$, the content of the inclusion $ B^X(x,2 \varepsilon_1)  \to X $ is trivial. By Theorem \ref{SW}, all we need to show is that for large enough $i$, $H_1^{\varepsilon_1}(X_i)$ has rank $\leq n-k $. By Corollary \ref{SW-gap}, there is a sequence $\rho_i \to 0$ with the property that $H_1^{\rho_i}(X_i) = H_1^{\varepsilon_1}(X_i)$ for each $i$. 

By Theorem \ref{KW}, there are $\varepsilon _ 0 > 0 $, $C \in \mathbb{N}$, subgroups $N_i \leq H_1(X_i)$, and a sequence $\delta_i \in [\varepsilon _ 0 , \varepsilon _1] $ with the property that for each $x \in X_i$, the content of the map $ B^{X_i}(x, \delta_i )  \to  X_i$ contains $N_i$ as a subgroup of index $\leq C$

Let $ x_1, \ldots , x_m  \in X$ be such that $X = \cup_{j=1}^m B^X(x_j, \varepsilon_0/3)$, and set $x^i_j : = h_i (x_j)$. Then for large enough $i$, the balls $B^{X_i}(x^i_j, \varepsilon_0/2)$ cover $X_i$. This implies that for large enough $i$, each ball of radius $\rho_i$ in $X_i$ is contained in a ball of the form $B^{X_i}(x^i_j, \varepsilon_0)$. Hence if we let $\mathcal{U}_i$ denote the family $\{ B^{X_i} ( x_j^i , \delta_i )   \}_{j=1}^m$, then by Lemma \ref{refinement} we get
\[ H_1^{\rho_i}(X_i) \leq  H_1 (\mathcal{U}_i \prec  X_i ) \leq H_1^{\varepsilon_1}(X_i) = H_1^{\rho_i} (X_i) .      \]

Since $ H_1^{\mathcal{U}_i} (X_i)$ is generated by the contents of the inclusions $  B^{X_i}(x^i_j, \delta_i)  \to  X_i$ with $j \in \{ 1, \ldots , m \}$, the index of $N_i $ in $H_1^{\mathcal{U}_i}(X_i)$ is at most $C^m$. Therefore, the rank of $H_1^{\varepsilon_1}(X_i)$ equals the rank of $N_i$ for all large enough $i$.

Let $\Gamma_i \leq H_1(X_i)$ denote the content of the inclusion $B^{X_i}(p_i, \varepsilon_2) \to X_i$. Since $\varepsilon_2 \geq \varepsilon_1$,  $\Gamma_i$ contains $N_i$, and since $\Gamma_i$ equals the content of the inclusion $B^{X_i}(p_i, \eta_i) \to X_i$, and $\eta_i \leq \varepsilon_0$ for large enough $i$, the index of $N_i$ in $\Gamma_i$ is finite. Hence Theorem \ref{main} will follow from the following claim.

\begin{center}
\textbf{Claim:} For large enough $i$, $\Gamma_i$ has rank $\leq n- k$.
\end{center}
Let $\lambda_i \to \infty$ be a sequence that diverges so slowly that $\lambda_i \eta_i \to 0$ and the sequence $(\lambda_i X_i , p_i)$ converges in the pointed Gromov--Hausdorff sense to $\mathbb{R}^{k}$. We can achieve this since $p$ is $k$-regular and $\eta_ i \to 0$. 

Let $(Y_i, q_i)$ denote the regular cover of $(\lambda X_i, p_i)$ with Galois group $H_1(X_i)$. By Theorem \ref{compactness} and Theorem \ref{equivariant-compactness}, we can assume that the sequence $(Y_i, q_i)$ converges in the pointed Gromov--Hausdorff sense to a proper geodesic space $(Y,q)$, and the groups $H_1(X_i)$ converge in the equivariant Gromov--Hausdorff sense to some closed group $\Gamma \leq Iso (Y)$. Since all elements of $H_1(X_i) \backslash \Gamma_i$ move $q_i$ at least $\varepsilon _2 \lambda_i$ away, the equivariant Gromov--Hausdorff limit of $\Gamma_i$ equals $\Gamma$ as well. Note that from the definition of equivariant Gromov--Hausdorff convergence, it follows that the $\Gamma_i$-orbits of $q_i$ converge in the pointed Gromov--Hausdorff sense to the $\Gamma$-orbit of $q$.

By Corollary \ref{splitting}, $Y$ splits isometrically as a product $\mathbb{R}^{k} \times Z$ with $Z$ a proper geodesic space of Hausdorff dimension $\leq n - k$, such that the $Z$-fibers coincide with the $\Gamma$-orbits. Since the topological dimension is always dominated by the Hausdorff dimension (\cite{hurewicz-wallman}, Chapter 7), the topological dimension of $Z$ is at most $n- k$. Then by Theorem \ref{zamora}, the rank of $\Gamma_i$ is at most $n- k$ for large enough $i$.
\end{proof}

\begin{proof}[Proof of Corollary \ref{tori}:]
Let $\varepsilon > 0 $ be given by Theorem \ref{torus-stability}.  By Theorem \ref{main}, $\beta_1(X) \geq k$, and the result follows.
\end{proof}

\end{document}